\documentclass[12pt]{article}

\usepackage[utf8]{inputenc} 
\usepackage{amsmath} 
\usepackage{amssymb} 
\usepackage{amsthm} 
\newtheorem{theorem}{Theorem}
\newtheorem{proposition}[theorem]{Proposition}
\newtheorem{corollary}[theorem]{Corollary}
\theoremstyle{definition}

\DeclareMathOperator{\interior}{int}

\DeclareMathOperator{\conv}{conv}

\newcommand{\href}[2]{{#2}}
\usepackage{color}
\usepackage{enumitem}
\usepackage[misc]{ifsym}
\usepackage{tikz}
\usetikzlibrary{arrows, shadings}
\tikzset{
>=stealth',
axis/.style={<->},
}
\usepackage{ifthen} 
\usepackage{tikz-3dplot} 
\usepackage{caption}
\newcommand{\R}{\mathbb{R}}
\renewcommand{\P}{\mathcal{P}}
\newcommand{\D}{\mathcal{D}}

\title{Geometric duality and parametric duality for multiple objective linear programs are equivalent}
\author{Daniel Dörfler\thanks{Friedrich Schiller University Jena, Germany, \href{mailto:daniel.doerfler@uni-jena.de}{daniel.doerfler@uni-jena.de}}\and Andreas Löhne\thanks{Friedrich Schiller University Jena, Germany, \href{mailto:andreas.loehne@uni-jena.de}{andreas.loehne@uni-jena.de}}}
\begin{document}
\maketitle
\begin{abstract}
\noindent
In 2011, Luc \cite{luc:on_duality_in_multiple_objective_linear_programming} introduced {\em parametric duality} for multiple objective linear programs. He showed that {\em geometric duality}, introduced in 2008 by Heyde and L\"ohne \cite{geom_dual}, is a consequence of parametric duality. We show the converse statement: parametric duality can be derived from geometric duality. We point out that an easy geometric transformation embodies the relationship between both duality theories. The advantages of each theory are discussed.


\end{abstract}

\section{Introduction}
Duality plays an important role in optimization from a theoretical point of view as well as an algorithmic one. 
There are several approaches to duality in multiple objective programming, among the earliest are \cite{kornbluth:duality, Roedder77, isermann:dual_pair, isermann:duality}.
A discussion and comparison of these and further approaches can be found in \cite{HamHeyLoeTamWin04, BotGraWan09, luc:on_duality_in_multiple_objective_linear_programming}. 

The concept of geometric duality introduced in \cite{geom_dual} can be motivated by several arguments. Many approaches to multiple objective programming duality suffer from a duality gap in case of the right-hand side of the constraints being zero ($b=0$), compare the discussion in \cite{HamHeyLoeTamWin04}. This duality gap could be closed in  \cite{HamHeyLoeTamWin04} at the price of a set-valued objective function of the dual program. The theory in \cite[Section 4.5.2]{BotGraWan09} has no duality gap and a vector-valued objective, but a non-convex dual problem. In contrast to that, the geometric dual problem is a {\em vector linear program}, which we consider to be a multiple objective linear program where the ordering cone is a polyhedral convex cone $C$ rather than the natural cone $\R^q_+$. While duality in multiple objective programming had little practical relevance in the past, geometric duality has been applied in several fields: Dual algorithms to solve multiple objective linear programs have been developed in \cite{benson_dual_variant}. The multiple objective linear programming software {\em Bensolve} uses the geometric dual program for data storing \cite{bensolve_paper}. In financial mathematics, superhedging portfolios in markets with transaction costs can be interpreted dually using geometric duality \cite{superhedging_portfolios}.

The key idea of geometric duality is to define a duality relation between the primal and dual problem on the basis of polarity of polyhedral convex sets. This is a natural generalization of linear programming duality: Let $p$ be the optimal value of a linear program, the primal problem, and let $d$ be the optimal value of its dual problem. If both problems are feasible, linear programming duality yields $p=d$. If $p$ is identified with the interval $\P:=[p,\infty)$ and $d$ is identified with the interval $\D:=(-\infty, d]$, then $p=d$ can be interpreted as a kind of polarity between $\P$ and $\D$. $\P$ is the intersection of half-spaces $H_+(1,w):=\{y \in \R \mid 1\cdot y \geq w\}$ over $w \in \D$. Likewise, $\D$ is the intersection of half-spaces $H_-(1,y):=\{w \in \R \mid 1 \cdot w \leq y\}$ over $y \in \P$. This duality relation is very simple for the one-dimensional case (its the same as $p=d$) since one-dimensional convex polyhedra are intervals. It becomes more difficult for higher dimensions as pointed out below.

The theories we consider in this paper are geometric duality by Heyde and Löhne \cite{geom_dual} and parametric duality as introduced by Luc in \cite[Section 4]{luc:on_duality_in_multiple_objective_linear_programming}. We point out that both approaches are equivalent in the sense that one can be easily derived from the other. Luc \cite{luc:on_duality_in_multiple_objective_linear_programming} already has shown that geometric duality is a consequence of parametric duality. In this paper (see also \cite{bachelor_thesis}) we show the converse statement. Both approaches have advantages: The parametric dual problem for a multiple objective linear program with $q$ objectives is based on the natural ordering cone but has $q+1$ objectives. The geometric dual problem has only $q$ objectives, as the primal problem, but the ordering cone has a special form.

\section{Notation and problem formulation}
For a set $M \subseteq \R^q$ and a pointed convex cone $C \subseteq \R^q$ an element $x \in M$ is called {\em $C$-minimal} if $(\{x\}-C \setminus\{0\}) \cap M = \emptyset$ and, if $C$ has nonempty interior, $\interior C \neq \emptyset$, $x \in M$ is called {\em weakly $C$-minimal} if $(\{x\}-\interior C) \cap M = \emptyset$.
We call a point $x \in M$ {\em (weakly) $C$-maximal} if $x$ is (weakly) $(-C)$-minimal. Throughout this paper we only consider two ordering cones: $$\R^q_+ := \{y \in \mathbb{R}^{q} \mid y_{1} \geq 0, \dots, y_{q} \geq 0 \}$$ 
and 
$$K := \{y \in \mathbb{R}^{q} \mid y_{1} = \dots = y_{q-1}=0, y_{q} \geq 0 \}.$$ We write $y \geq_{\R^q_+} x$ (or $y\geq x$) if $y-x \in \R^q_+$ and $y \geq_K x$ if $y-x \in K$.

A convex subset $F$ of a polyhedral convex set $M\subseteq \R^q$ is called a {\em face} of $M$ if
$$ [\lambda \in (0,1),\, x,y \in M,\, \lambda x + (1-\lambda)y \in F] \implies x,y \in F.$$ 
The $(r-1)$-dimensional faces of an $r$-dimensional polyhedral convex set $M$ are called {\em facets}; $0$-dimensional faces are called {\em vertices}. A face $F$ of $M$ is called proper if $\emptyset \neq F \neq M$. 

For matrices $P \in \mathbb{R}^{q \times n}$, $A \in \mathbb{R}^{m \times n}$ and a vector $b \in \mathbb{R}^{m}$ we consider the following {\em multiple objective linear program}:
\begin{equation}\label{P}
\textstyle\min Px \quad \text{s.t.} \quad Ax \geq b. \tag{P}
\end{equation}
For a given problem \eqref{P} with the feasible region $S = \{x \in \mathbb{R}^{n} \mid Ax \geq b\}$ we call the set
\begin{equation*}
\mathcal{P} := P[S] + \R^q_+ := \{y \in \R^q |\; \exists x \in S:\; y \geq_{\R^q_+} Px \}
\end{equation*}
{\em upper image} of \eqref{P}.

\section{Geometric duality}
In this section we recall the main result of geometric duality.
We define the dual linear objective function by
$$D:\mathbb{R}^{m} \times \mathbb{R}^{q} \to \mathbb{R}^{q},\; D(u,w) := (w_{1},\dots,w_{q-1},b^{\intercal}u)$$
and consider the following dual vector optimization problem:
\begin{equation}\label{D}
\begin{aligned}
\textstyle{\max_{K}} D(u,w) \quad \text{s.t.} \quad & A^{\intercal}u = P^{\intercal}w, \\
& (u,w) \geq 0 \\
& e^{\intercal}w=1,
\end{aligned}
\tag{D}
\end{equation}
where we set $e=(1,\dots,1)^{\intercal}$.
Analogously to the primal case, we define the {\em lower image} of the dual problem as 
$$\mathcal{D} := D[T]-K := \{y \in \R^q |\; \exists (u,w) \in T:\; y \leq_K D(u,w))\}$$
with $T$ being the feasible region of \eqref{D}, i.e. 
$$T = \{(u,w) \in \mathbb{R}^{m} \times \mathbb{R}^{q} \mid A^{\intercal}u = P^{\intercal}w, (u,w) \geq 0, e^{\intercal}w = 1 \}.$$
To formulate the main result of geometric duality we use the {\em coupling function}:
\begin{equation*}
\varphi: \mathbb{R}^{q} \times \mathbb{R}^{q} \to \mathbb{R}^{q},\qquad \varphi (y,y^{*}) := \sum_{i=1}^{q-1}y_{i}y_{i}^{*}+y_{q}\left(1-\sum_{i=1}^{q-1}y_{i}^{*}\right)-y_{q}^{*}.
\end{equation*}
For $y,y^{*} \in \mathbb{R}^{q}$ we define the sets
\begin{equation*}
H^{*}(y) := \{y^{*} \in \mathbb{R}^{q} \mid \varphi(y,y^{*}) = 0\} \text{ and } H(y^{*}) := \{y \in \mathbb{R}^{q} \mid \varphi(y,y^{*}) = 0. \}
\end{equation*}
Since $\varphi(y,\cdot)$ and $\varphi(\cdot,y^{*})$ are affine functions, the sets $H^{*}(y)$ and $H(y^{*})$ describe hyperplanes in $\mathbb{R}^{q}$. The duality mapping is defined as
\begin{equation*}
\Psi : 2^{\mathbb{R}^q} \to 2^{\mathbb{R}^q},\qquad \Psi(F^{*}) := \bigcap_{y^{*} \in F^{*}} H(y^{*}) \cap \mathcal{P}.
\end{equation*}
A face is called $K$-maximal (weakly $\R^q_+$-minimal) if it consists of only $K$-maximal (weakly $\R^q_+$-minimal) elements.
We can now state the main result of geometric duality, which shows that $\Psi$ defines a duality relation between $\mathcal{P}$ and $\mathcal{D}$.

\begin{theorem}\label{th_geo}(\cite[Theorem 3.1]{geom_dual})
$\Psi$ is an inclusion-reversing \textup{(}i.e. $F_1 \subseteq F_2 \Rightarrow \Psi(F_2) \subseteq \Psi(F_1)$\textup{)} one-to-one map between the set of all $K$-maximal proper faces of $\mathcal{D}$ and the set of all weakly $\R^q_+$-minimal proper faces of $\mathcal{P}$. The inverse map is given by
\begin{equation*}
\Psi^{-1}(F) = \bigcap_{y \in F} H^{*}(y) \cap \mathcal{D}.
\end{equation*}
Moreover, for every $K$-maximal proper face $F^{*}$ of $\mathcal{D}$, one has 
$$\dim F^{*} + \dim\Psi(F^{*})=q-1.$$
\end{theorem}

Note that $\Psi$ maps the $K$-maximal vertices of $\mathcal{D}$ to the weakly $\R^q_+$-minimal facets of $\mathcal{P}$ and $\Psi^{-1}$ maps the weakly $\R^q_+$-minimal vertices of $\mathcal{P}$ to the $K$-maximal facets of $\mathcal{D}$. We illustrate geometric duality by an example. Consider problem \eqref{P} with the following data:
\[
P = 
\begin{pmatrix}
1 & 0 \\
0 & 1
\end{pmatrix},
\quad
A =
\begin{pmatrix}
1 & -1 \\
8 & 2 \\
4 & 2 \\
2 & 4 
\end{pmatrix},
\quad
b =
\begin{pmatrix}
-3 \\
11 \\
7 \\
5
\end{pmatrix}.
\]
The set $\mathcal{D}$ can be calculated as:
\[
\mathcal{D}=\conv\left\{
\begin{pmatrix}
\frac{1}{3} \\ \frac{5}{6}
\end{pmatrix}
,
\begin{pmatrix}
\frac{2}{3} \\ \frac{7}{6}
\end{pmatrix}
,
\begin{pmatrix}
\frac{4}{5} \\ \frac{11}{10}
\end{pmatrix}
,
\begin{pmatrix}
1 \\ \frac{1}{2}
\end{pmatrix}
\right\}-K,
\]
where $\conv M$ denotes the convex hull of a set $M$. The upper and lower images of the primal and dual problem for this example are shown in Figure~\ref{fig1}.

\begin{figure}
\begin{center}
\begin{minipage}[t]{.45\linewidth}
\begin{tikzpicture}
\coordinate (y) at (0,5);
\coordinate (x) at (5,0);
\draw[axis,thick] (y) node[above] {$x_{2}$} [thick] -- (0,0) [thick] -- (x) node[right] {$x_{1}$};
\node[left=1pt] at (0,1) {1};
\node[below=1pt] at (1,0) {1};

\shade[lower left=gray, upper right=white] (0.5,5.5) -- (0.5,3.5) -- (1,1.5) -- (1.5,0.5) -- (3.5,-0.5) -- (5,-0.5) -- (5,5.5) -- cycle;
\draw[thick] (0.5,5.5) -- (0.5,3.5) -- (1,1.5) -- (1.5,0.5) -- (3.5,-0.5);
\draw[axis,thick] (y) node[above] {$x_{2}$} [thick] -- (0,0) [thick] -- (x) node[right] {$x_{1}$};
\node[xshift=5pt, yshift=10pt] at (2.5,0) {$\mathcal{P}$};
\end{tikzpicture}
\end{minipage}
\begin{minipage}[t]{.45\linewidth}
\begin{tikzpicture}[scale=3.33]
\coordinate (y) at (0,1.5);
\coordinate (x) at (1.5,0);
\node[left=1pt] at (0,1) {1};
\node[below=1pt, xshift=5pt] at (1,0) {1};

\shade[top color=gray, bottom color=white] (0.33,-0.2) -- (0.33,0.833) -- (0.666,1.166) -- (0.8,1.1) -- (1,0.5) -- (1,-0.2) -- cycle;
\draw[thick] (0.33,-0.2) -- (0.33,0.833) node[left, xshift=3pt] {$\left(\frac{1}{3},\frac{5}{6}\right)$} -- (0.666,1.166) node[above] {$\left(\frac{2}{3},\frac{7}{6}\right)$} -- (0.8,1.1) node[right] {$\left(\frac{4}{5},\frac{11}{10}\right)$} -- (1,0.5) node[right, xshift=-2pt] {$\left(1,\frac{1}{2}\right)$} -- (1,-0.2);
\draw[axis,thick] (y) node[above] {$b^{\intercal}u$} [thick] -- (0,0) [thick] -- (x) node[right] {$w_{1}$};
\node[xshift=-10pt, yshift=10pt] at (1,0) {$\mathcal{D}$};
\end{tikzpicture}
\end{minipage}
\end{center}
\caption{The four weakly $\R^2_+$-minimal facets of $\mathcal{P}$ correspond to the four $K$-maximal vertices of $\mathcal{D}$ and the three weakly $\R^2_+$-minimal vertices of $\mathcal{P}$ correspond to the three $K$-maximal facets of $\mathcal{D}$.}\label{fig1}
\end{figure}
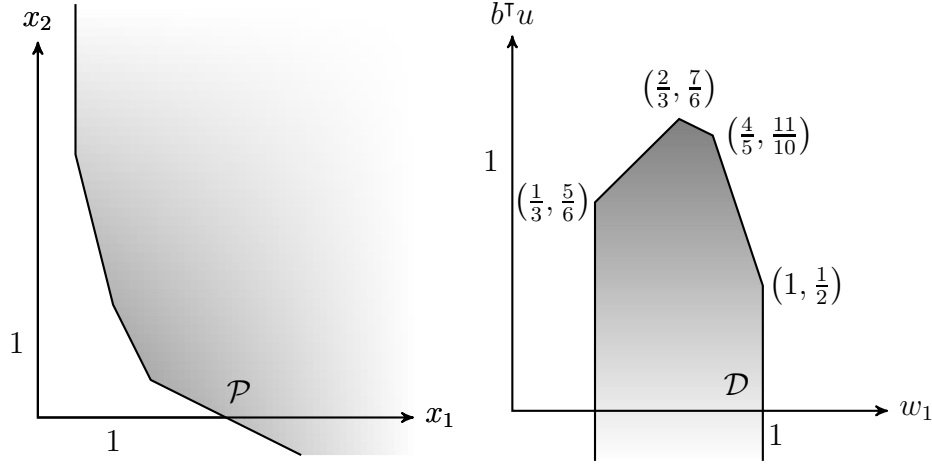

\section{Parametric duality}\label{3_2}

In this section we recall parametric duality as introduced in \cite{luc:on_duality_in_multiple_objective_linear_programming}. We restate the main result of \cite{luc:on_duality_in_multiple_objective_linear_programming} by the same notation as used for geometric duality.
In parametric duality, the geometric dual objective function 
$$D:\mathbb{R}^{m} \times \mathbb{R}^{q} \to \mathbb{R}^{q},\; D(u,w) := (w_{1},\dots,w_{q-1},b^{\intercal}u)$$
is replaced by
$$\bar D:\mathbb{R}^{m} \times \mathbb{R}^{q} \to \mathbb{R}^{q+1},\; \bar D(u,w) := (w_{1},\dots,w_{q},b^{\intercal}u).$$
Moreover the ordering cone $K$ used for the geometric dual is replaced by $\R^{q+1}_+$. This leads to the following dual multiple objective linear program \eqref{D*} of \eqref{P}:
\begin{equation}\label{D*}
\begin{aligned}
\textstyle{\max_{\R^{q+1}_+}} \bar D(u,w)\quad \text{s.t.} \quad & A^{\intercal}u = P^{\intercal}w, \\
& (u,w) \geq 0, \\
& e^{\intercal}w=1.
\end{aligned}
\tag{$\bar{{\rm D}}$}
\end{equation}
Note that \eqref{D*} has the same feasible set as \eqref{D}, denoted by $T$.
The lower image of \eqref{D*} is the set
$$ \bar{\mathcal{D}} := \bar D[T]-\R^{q+1}_+ := \{y \in \R^{q+1}|\; \exists (u,w) \in T:\; y \leq_{\R^{q+1}_+} \bar{D}(u,w)\}.$$
A face of a polyhedral convex set $M$ is said to be $\R^{q+1}_+$-minimal / $\R^{q+1}_+$-maximal if it only consists of $\R^{q+1}_+$-minimal / $\R^{q+1}_+$-maximal elements of $M$. The main result of parametric duality can be stated as follows.

\begin{theorem}\label{th_para}(\cite[Corollary 4.4]{luc:on_duality_in_multiple_objective_linear_programming})
A face $F$ of $\mathcal{P}$ is (weakly) $\R^{q}_+$-minimal if and only if the set
\begin{equation*}
\bar F^{*} := \bigcap_{y \in F} \left\lbrace
\begin{pmatrix}
w \\ t
\end{pmatrix}
\in \bar{\mathcal{D}} \,\middle\vert\,
\begin{pmatrix}
y \\ -1
\end{pmatrix}^{\intercal}
\begin{pmatrix}
w \\ t
\end{pmatrix}
=0\right\rbrace,
\end{equation*}
in which at least one $w$ is strictly positive (respectively $w \geq 0$), is a $\R^{q+1}_+$-maximal face of $\bar{\mathcal{D}}$.

Similarly a face $\bar F^{*}$ of $\bar{\mathcal{D}}$ is $\R^{q+1}_+$-maximal if and only if the set
\begin{equation*}
F := \bigcap_{
(w^{\intercal}, t)^{\intercal}
\in \bar F^{*}} \left\lbrace y \in \mathcal{P} \,\middle\vert\,
\begin{pmatrix}
y \\ -1
\end{pmatrix}^{\intercal}
\begin{pmatrix}
w \\ t
\end{pmatrix}
=0\right\rbrace,
\end{equation*}
in which at least one $w$ is strictly positive (respectively $w \geq 0$), is a (weakly) $\R^{q}_+$-minimal face of $\mathcal{P}$.
\end{theorem}
In Figure \ref{fig_2} we illustrate parametric duality with the example from the previous section.
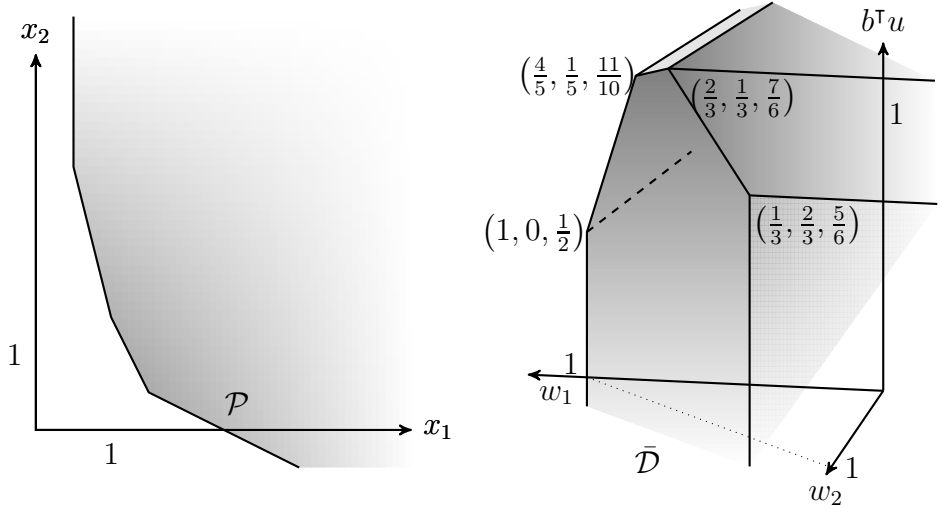
\begin{figure}[ht]
\begin{center}
\begin{minipage}[t]{.45\linewidth}
\begin{tikzpicture}[scale=1]
\coordinate (y) at (0,5);
\coordinate (x) at (5,0);
\draw[axis,thick] (y) node[above] {$x_{2}$} [thick] -- (0,0) [thick] -- (x) node[right] {$x_{1}$};
\node[left=1pt] at (0,1) {1};
\node[below=1pt] at (1,0) {1};

\shade[lower left=gray, upper right=white] (0.5,5.5) -- (0.5,3.5) -- (1,1.5) -- (1.5,0.5) -- (3.5,-0.5) -- (5,-0.5) -- (5,5.5) -- cycle;
\draw[thick] (0.5,5.5) -- (0.5,3.5) -- (1,1.5) -- (1.5,0.5) -- (3.5,-0.5);
\draw[axis,thick] (y) node[above] {$x_{2}$} [thick] -- (0,0) [thick] -- (x) node[right] {$x_{1}$};
\node[xshift=5pt, yshift=10pt] at (2.5,0) {$\mathcal{P}$};

\node at (0,-1) {};
\end{tikzpicture}
\end{minipage}
\begin{minipage}[t]{.45\linewidth}
\tdplotsetmaincoords{75}{190}
\begin{tikzpicture}[tdplot_main_coords, scale=4]
\shade[top color=gray, bottom color=white] (1,0,-0.1) -- (1,0,0.5) -- (0.8, 0.2, 1.1) -- (0.66667, 0.33333, 1.166667) -- (0.33333, 0.66667, 0.833333) -- (0.33333, 0.66667, -0.1);



\fill[color=gray, opacity=0.2] (0.8, 0.2, 1.1) -- (0.66667, 0.33333, 1.166667) -- (0.66667, 0.33333, 1.166667) -- (0.66667, -1.66667, 0.866667) -- (0.8, -1.8, 0.8);
\shade[left color=gray, right color=white] (0.66667, 0.33333, 1.166667) -- (-0.233333, 0.33333, 1.166667) -- (0.66667, -1.66667, 0.866667);
\shade[left color=gray, right color=white] (0.66667, 0.33333, 1.166667) -- (0.33333, 0.66667, 0.833333) -- (-0.3, 0.66667, 0.833333) -- (-0.233333, 0.33333, 1.166667);

\shade[upper left=gray, lower right=white, opacity = 0.5] (0.33333, 0.66667, -0.1) -- (0.33333, 0.66667, 0.833333) -- (-0.3, 0.66667, 0.833333);
\draw[thick] (1, 0, -0.1) -- (1,0,0.5) -- (0.8, 0.2, 1.1) -- (0.66667, 0.33333, 1.166667);
\draw[thick] (0.66667, 0.33333, 1.166667) -- (0.33333, 0.66667, 0.833333) -- (0.33333, 0.66667, -0.1);



\draw[thick] (0.66667, 0.33333, 1.166667) -- (-0.233333, 0.33333, 1.166667);

\draw[thick] (0.33333, 0.66667, 0.833333) -- (-0.3, 0.66667, 0.833333);

\draw[dashed, thick] (1,0,0.5) -- (1,-2,0.25);
\draw[thick] (0.8, 0.2, 1.1) -- (0.8, -1.8, 0.8);
\draw[thick] (0.66667, 0.33333, 1.166667) -- (0.66667, -1.66667, 0.866667);

\node[xshift=-20pt, yshift=0pt] at (1, 0, 0.5) {$\left(1,0,\frac{1}{2}\right)$};

\node[xshift=-22pt, yshift=0pt] at (0.8, 0.2, 1.1) {$\left(\frac{4}{5},\frac{1}{5},\frac{11}{10}\right)$};

\node[xshift=28pt, yshift=-11pt] at (0.66667, 0.33333, 1.166667) {$\left(\frac{2}{3},\frac{1}{3},\frac{7}{6}\right)$};

\node[yshift=-11pt, xshift=22pt] at (0.33333, 0.66667, 0.833333) {$\left(\frac{1}{3},\frac{2}{3},\frac{5}{6}\right)$};

\node[yshift=8pt, xshift=15pt] at (0.7, 1.3, 0) {$\bar{\mathcal{D}}$};

\draw[thick,->] (0,0,0) -- (1.2,0,0) node[anchor=north west]{$w_{1}$};
\draw[thick,->] (0,0,0) -- (0,1.1,0) node[anchor=north]{$w_{2}$};
\draw[thick,->] (0,0,0) -- (0,0,1.2) node[anchor=south]{$b^{\intercal}u$};
\draw[dotted] (1,0,0) -- (0,1,0);
\node[yshift=5pt, xshift=-6pt] at (1,0,0) {1};
\node[yshift=0pt, xshift= 8pt] at (0,1,0) {1};
\node[xshift=5pt, yshift=-5pt] at (0,0,1) {1};
\end{tikzpicture}
\end{minipage}
\end{center}
\caption{The four weakly $\R^2_+$-minimal facets of $\mathcal{P}$ correspond to the four $\R^3_+$-maximal vertices of $\bar{\mathcal{D}}$. Likewise the three weakly $\R^2_+$-minimal vertices of $\mathcal{P}$ correspond to the three $\R^3_+$-maximal faces of $\bar{\mathcal{D}}$, which are the three bounded one-dimensional faces of $\bar{\mathcal{D}}$. Observe that $\bar{\mathcal{D}}$ intersected with the hyperplane $\{w \in \mathbb{R}^{3} \mid w_{1}+w_{2} = 1\}$ has the same facial structure as $\mathcal{D}$.}
\label{fig_2}
\end{figure}

\section{Equivalence between geometric and parametric duality}

In this last section we show that geometric duality and parametric duality are equivalent. This means that Theorem \ref{th_para} can be derived from Theorem \ref{th_geo}, and vise versa, where the latter has already been shown by Luc \cite{luc:on_duality_in_multiple_objective_linear_programming}.

Let $\pi:\R^{q+1}\to \R^q$ denote the projection
$$\pi(w_1,\dots,w_{q+1}):=(w_1,\dots,w_{q-1},w_{q+1}).$$
 For a set $F \subseteq \R^{q}$, we set $\pi^{-1}[F]:=\{w \in \R^{q+1}|\; \pi(w) \in F\}$. The following proposition describes the essential geometric relation between geometric duality and parametric duality.

\begin{proposition}\label{prop_1}
The function $\Phi:2^{\R^q}\to 2^{\R^{q+1}}$, 
$$ \Phi:2^{\R^q}\to 2^{\R^{q+1}},\quad \Phi(F):=\pi^{-1}[F] \cap \{(w,t) \in \R^{q} \times \R |\;e^{\intercal} w = 1\}$$
is an inclusion-invariant one-to-one map between the $K$-maximal faces of $\mathcal{D}$ and the $\R^{q+1}_+$-maximal faces of $\bar{\mathcal{D}}$. The inverse map is 
$$ \Phi^{-1}:2^{\R^{q+1}}\to 2^{\R^{q}},\quad \Phi^{-1}(\bar F) = \pi[\bar F].$$
\end{proposition}
\begin{proof}
Consider the affine function 
$$\gamma:\R^q\to \R^{q+1},\quad \gamma(w) :=\left(w_1,\dots,w_{q-1},1-\sum_{i=1}^{q-1} w_i,\,w_q \right)$$
and note that $\Phi(F)=\{\gamma(w)|\, w \in F\}$.
The map $\gamma:\R^q\to E := \{(w,t) \in \R^q \times \R |\;e^{\intercal} w = 1\}$ is bijective with inverse map $\pi |_{E}$. Hence, $\Phi:2^{\R^q} \to 2^{E}$ defined by $\Phi(F)=\gamma[F]=\pi^{-1}[F] \cap E$ is one-to-one with inverse mapping $\Phi^{-1}(\bar F)=\pi[\bar F]=\gamma^{-1}[\bar F]$. Moreover, the set of $\R^{q+1}_+$-maximal points in $\bar{\mathcal{D}}$ is a subset of $E$.
From the definitions of the lower images $\mathcal{D}$ and $\bar{\mathcal{D}}$, we obtain 
$$\bar{\mathcal{D}} = \Phi(\mathcal{D}) - (\R^q_+ \times \{0\})$$
and 
$$\mathcal{D} = \Phi^{-1}(\bar{\mathcal{D}} \cap E ).$$
It follows that a point $w$ is $K$-maximal in $\mathcal{D}$ if and only if $\gamma(w)$ is $\R^{q+1}_+$-maximal in $\bar{\mathcal{D}}$. To see this, note that the definition of $\gamma$ implies
$$\bar w \leq_K w  \iff \gamma(\bar w) \leq_{\R^{q+1}_+} \gamma(w),$$
and, since $\gamma(w) \in E$,
 $$w \in \mathcal{D} \iff \gamma(w) \in \mathcal{D}'.$$

It remains to show that $K$-maximal faces of $\mathcal{D}$ are mapped to $\R^{q+1}_+$-maximal faces of $\bar{\mathcal{D}}$ and vice versa. If $F$ is a face of $\mathcal{D}$, then $\Phi(F)$ is a face of $\Phi(\mathcal{D})$. Since $E$ is a supporting hyperplane of $\mathcal{\bar{D}}$ with $\Phi(\mathcal{D})=\mathcal{\bar{D}} \cap E$, $\Phi(\mathcal{D})$ is a face of $\mathcal{\bar{D}}$ and hence $\Phi(F)$ is a face of $\bar{\mathcal{D}}$, which implies the first implication.

Now let $\bar F$ be an $\R^{q+1}_+$-maximal face of $\bar{\mathcal{D}}$. Then $\bar F$ belongs to the hyperplane $E$ and thus $\Phi^{-1}(\bar F)= \pi[\bar F]$ is a face of $\mathcal{D}$.
\end{proof}

From the geometric duality theorem (Theorem \ref{th_geo}) and Proposition \ref{prop_1} we deduce the following variant of parametric duality. Its formulation is analogous to the one of the geometric duality theorem.

\begin{corollary}\label{th_para2}
	$\Psi\circ\Phi^{-1}$ is an inclusion-reversing one-to-one map between the set of all $\R^{q+1}_+$-maximal proper faces of $\bar{\mathcal{D}}$ and the set of all weakly $\R^q_+$-minimal proper faces of $\mathcal{P}$. The inverse map is $\Phi\circ \Psi^{-1}$.
	Moreover, for every $\R^{q+1}_+$-maximal proper face $\bar F^{*}$ of $\bar{\mathcal{D}}$, one has
	$$\dim \bar F^{*} + \dim (\Psi \circ\Phi^{-1})(\bar F^{*})=q-1.$$
\end{corollary}
\begin{proof}
	Follows from Theorem \ref{th_geo} and Proposition \ref{prop_1}. For the last statement, take into the account that $\Phi$ preserves the dimension of faces of $\mathcal{D}$.
\end{proof}

Finally we prove Theorem \ref{th_para} by using Theorem \ref{th_geo}:

\begin{proof}
We have
\begin{equation*}
(\Phi\circ\Psi^{-1})(F) = \bigcap_{y \in F} \left\lbrace
\begin{pmatrix}
w \\ t
\end{pmatrix}
\in \bar{\mathcal{D}} \,\middle\vert\,
\begin{pmatrix}
y \\ -1
\end{pmatrix}^{\intercal}
\begin{pmatrix}
w \\ t
\end{pmatrix}
=0\right\rbrace
\end{equation*}
and
\begin{equation*}
(\Psi\circ\Phi^{-1})(\bar F^*) := \bigcap_{
(w^{\intercal}, t)^{\intercal}
\in \bar F^{*}} \left\lbrace y \in \mathcal{P} \,\middle\vert\,
\begin{pmatrix}
y \\ -1
\end{pmatrix}^{\intercal}
\begin{pmatrix}
w \\ t
\end{pmatrix}
=0\right\rbrace.
\end{equation*}	
Thus the first part of Theorem \ref{th_para} follows from Corollary \ref{th_para2} (and hence from Theorem \ref{th_geo}). It remains to show that weakly $\R^q_+$-minimal faces $F$ of $\mathcal{P}$ are even $\R^q_+$-minimal if and only if at least one $w$ with $(w,t) \in \bar F^*:=(\Phi\circ\Psi^{-1})(F)$ is positive in each component, denoted $w>0$. 

We have $(w,t) \in \bar F^*$ if and only if $F=(\Psi\circ\Phi^{-1})(\bar F^*)$ belongs to the face $\hat F:=\{y \in \R^q|\; w^{\intercal} y = t\}\cap \P$ of $\mathcal{P}$ with $e^{\intercal} w = 1$. Since $\P = \P + \R^q_+$, the elements of $\hat F$ are $\R^q_+$-minimal if and only if $w>0$.
\end{proof}

\noindent
{\bf Acknowledgements:} The authors thank Andreas H. Hamel for bringing the subject to their attention.


\end{document}